\documentclass[a4paper,12pt]{article}

\usepackage{amsmath}
\usepackage{amssymb}
\usepackage{amsthm}
\usepackage{cite}
\usepackage{color}

\newcommand{\la}{\langle}
\newcommand{\ra}{\rangle}
\newcommand{\pr}{\partial}
\newcommand{\dom}{\Omega}
\newcommand{\sig}{\Sigma}
\newcommand{\gam}{\Gamma}
\newcommand{\R}{\mathbb{R}}

\newcommand{\supp}{\mbox{supp\;}}
\newcommand{\hb}{H^{1/2}_{co}}
\newcommand{\hbi}{H^{-1/2}_{co}}
\newcommand{\DN}{{Dirichlet-to-Neumann }}
\newcommand{\dn}[2]{{\Lambda^{#1}_{#2}\;}}
\newcommand{\Ci}{C^{(1)}}
\newcommand{\Cii}{C^{(2)}}
\newcommand{\Cz}{{C^0}}

\newtheorem{thm}{Theorem}[section]
\newtheorem{prop}{Proposition}[section]
\newtheorem{lem}{Lemma}[section]
\newtheorem{remark}{Remark}[section]

\title{Uniqueness in the inverse boundary value problem for piecewise homogeneous anisotropic elasticity}

\author{C\u at\u alin I. C\^arstea\thanks{National Center for Theoretical Sciences, Taipei 10617, Taiwan} \and Naofumi Honda\thanks{Hokkaido University, Sapporo  060-0808, Japan} \and  Gen Nakamura\thanks{Hokkaido University, Sapporo  060-0808, Japan}
}

\date{}

\begin{document}

\maketitle

\begin{abstract}
Consider a three dimensional piecewise homogeneous anisotropic elastic medium $\Omega$ which is a bounded domain consisting of a finite number of bounded  subdomains $D_\alpha$, with each $D_\alpha$  a homogeneous elastic medium. One typical example is a finite element model with elements with curvilinear interfaces for an ansiotropic elastic medium. Assuming the $D_\alpha$ are known and Lipschitz, we are concerned with the uniqueness in the inverse boundary value problem of identifying the anisotropic elasticity tensor on  $\Omega$ from a localized Dirichlet to Neumann map given on a part of the boundary $\partial D_{\alpha_0}\cap\partial\Omega$ of $\partial\Omega$ for a single $\alpha_0$, where $\partial D_{\alpha_0}$ denotes the boundary of $ D_{\alpha_0}$. If we can connect each $D_\alpha$ to $D_{\alpha_0}$ by a chain of $\{D_{\alpha_i}\}_{i=1}^n$ such that interfaces between adjacent regions contain a curved portion, we obtain global uniqueness for this inverse boundary value problem.  If the $D_\alpha$ are not known but are subanalytic subsets of $\R^3$ with curved boundaries, then we also obtain global uniqueness. 
\end{abstract}



\section{Introduction}

Generally, an inverse boundary value problem considers the question of determining the interior physical properties of a medium from measurements made on the boundary of that medium. Such a problem, which is of clear practical interest, has received a lot of attention from the mathematics community since the publishing of the seminal paper \cite{Calderon} on the inverse conductivity problem. 

One important aspect of inverse boundary value problems is the uniqueness problem, namely the question of whether the  measurements on the boundary (all possible measurements or a proper subset of them) can determine the interior material properties of the medium uniquely. In this paper we investigate this question for the case of three dimensional anisotropic elastic media which are piecewise homogeneous. We prove that, under some restrictions on the geometry of the interfaces between the homogeneous pieces of the medium, many measurements done on part of the boundary uniquely determine the elastic tensor in the interior for the following two cases. One is for known Lipschitz smooth interfaces with curved portions. The other is for unknown homogeneous pieces which are subanalytic sets with curved boundaries. We may see our case as a very natural inverse problem for a finite element model with curvilinear polyhedral grids of an anisotropic elastic medium to identify its elastic properties by many measurements on a part of its boundary. Further, we would like to emphasize that we don't need to know the symmetry axis or symmetry planes for the piecewise homogeneous anisotropic media. Indeed, we do not need to make any symmetry assumption.

Usually, for the former case, the argument for determining the material coefficients consists of two steps. The first is to identify the coefficients at the boundary, and the second to determine the coefficients in the interior. For convenience we refer to them as {\it boundary determination} and {\it interior determination}. There are numerous other examples in the literature of works addressing these topics. Boundary determination has been studied extensively, for anisotopic elasticity past works include \cite{nakamura-tanuma2}, \cite{nakamura-tanuma-uhlmann}, \cite{tanuma}. Interior determination  has also been studied for various equations for example in
\cite{Alessandrini-de_Hoop-Gaburro}, \cite{Alessandrini-Vessella}, \cite{MR3035477}, \cite{MR2832161}, \cite{MR3291119}, \cite{MR3295939}, which consider piecewise problems with piecewise constant material coefficients with a set up very similar to our own. \cite{MR3291119} and \cite{MR3295939} consider the problem of elasticity, but in the isotropic case. We remark that the argument works well even in the case the interfaces with curved portions are boundaries of subanalytic sets.

The latter case can also be handled using boundary determination and interior determination if for given two piecewise homogeneous anisotropic media, we can have a common regions of homogeneity such that each of their interfaces has a curved portion. In \cite{Kohn-Vogelius2} the authors proved the existence of common regions for the conductivity equation with piecewise analytic isotropic conductivity, which was only given in the two dimensional case for ease of exposition. Although the existence of common regions seems true, we have not observed any proof so far for the higher dimensional case. We used the theory of subanalytic sets to prove the existence of common regions for the higher dimensional case.


The proofs of our main results in the aforementioned two cases use the boundary determination (cf. Proposition \ref{bd-thm}) and interior determination (cf. Proposition \ref{inner extension}). In the former case these two are enough, but in the latter case we need to give the existence of aforementioned common regions for given two piecewise homogeneous anisotropic media. The idea
for the boundary determination is to use the link between the surface impedance tensor and the fundamental solution of anisotropic homogeneous elastic equations (cf. \cite{nakamura-tanuma}, \cite{tanuma}).  Concerning the former case, in a series of recent papers 
(\cite{Alessandrini-de_Hoop-Gaburro}, \cite{Alessandrini-Kim}, \cite{Alessandrini-Vessella}, \cite{MR3035477}, \cite{MR2832161}, \cite{MR3291119}, \cite{MR3295939}) a method using Green functions (or singular solutions similar to these) is employed for the purpose of proving interior uniqueness. However, for elliptic systems, Green functions have not yet been well studied and it is hard to use Green functions for proving the interior determination. To overcome this difficulty we we have chosen to  adapt an inner extension of the DN map argument given in \cite{ikehata} to our case. 

For the proof of the existence of the common regions, we use the theory of subanalytic sets. Once having the common regions $\{\tilde{D}_\gamma\}$ such that $\Sigma\subset \partial \tilde{D}_{\gamma_0}\cap\partial\Omega$, then for any $\tilde{D}_\gamma$ with $\gamma\not=\gamma_0$ we can connect $\tilde{D}_\gamma$ by a chain $\{\tilde{D}_{\gamma_j}\}_{j=1}^N$ with $\tilde{D}_{\gamma_1}=\tilde{D}_{\gamma_0}$, $\tilde{D}_{\gamma_N}=\tilde{D}_{\gamma}$ and take a path going through this chain avoiding any singularities of $\tilde{D}_{\gamma_j}$, $j=1,\cdots, N$. The previous argument for the former case using the boundary determination and interior determination can be applied along a tubular neighborhood of this path.

We would like to mention that we were particularly motivated  by \cite{Alessandrini-de_Hoop-Gaburro} for the boundary determination, \cite{ikehata} for the inner extension of \DN  map (cf. Proposition \ref{inner extension}) and \cite{Kohn-Vogelius2} for the common regions.

The rest of this paper is organized as follows. In Section \ref{section-setup} we present the exact setup and the main results of this paper. 
Section \ref{section-boundary} is devoted to the proof of
Proposition \ref{bd-thm}. Then in Section \ref{section-proofs} we give the proof of Theorems \ref{main-thm} and \ref{main-thm-2} by proving Proposition \ref{inner extension} and combining it with Proposition \ref{bd-thm}.

\section{Setup and main results}\label{section-setup}

Let $\dom\subset\R^3$ be a  open bounded connected domain. We consider an elastic tensor $C=C(x)=(C_{ijkl}(x))$ defined for $x=(x_1,x_2,x_3)\in\overline\dom$ with real valued functions $C_{ijkl}(x)$ defined on $\overline\Omega$, which satisfies the symmetries
\begin{equation}
C_{ijkl}(x)=C_{ijlk}(x),\quad C_{ijkl}(x)=C_{klij}(x),\quad x\in\overline\Omega,\,\, i,j,k,l\in\{1,2,3\},
\end{equation}
and the strong convexity condition, i.e. there exists $\lambda>0$ such that for any symmetric matrix $\epsilon=(\epsilon_{ij})$,
\begin{equation}
\epsilon:(C::\epsilon)=\sum_{i,j,k,l=1}^3 C_{ijkl}(x)\epsilon_{ij}\epsilon_{kl}\ge \lambda (\epsilon:\epsilon),\,\,x\in\overline\Omega,
\end{equation}
where $\epsilon:\eta$ is the inner product of matrices $\epsilon$ and $\eta=(\eta_{ij})$ defined by
$\epsilon:\eta=\sum_{i,j=1}^3\epsilon_{ij}\eta_{ij}$, and $C::\eta$ is a matrix whose $(i,j)$ component $(C::\eta)_{ij}$ is defined by
$(C::\eta)_{ij}=\sum_{k,l=1}^3 C_{ijkl}(x)\eta_{kl}$. 

We further assume that $C$ is piecewise homogeneous, that is that there are a finite number of open, connected,  subdomains $D_\alpha$, $\alpha\in A$, such that $\bar\dom=\cup_{\alpha\in A}\bar D_\alpha$, $D_\alpha\cap D_\beta=\emptyset$ if $\alpha\neq\beta$, and $C$ is constant on each $D_\alpha$.
In the rest of this paper all
elastic tensors are assumed to satisfy the symmetry, strong convexity condition and piecewise homogeneity.

Below we will consider two cases: 
\begin{itemize}
\item[{\it i)}]  $\dom$ and each $D_\alpha$ are Lipschitz domains. 
\item[{\it ii)}] $\dom$ and each $D_\alpha$ are open subanalytic subsets of $\R^3$. 
\end{itemize}

The second case requires some elaboration: Let us briefly recall the definition of a subanalytic set and its major properties. 
We first recall the one of a semi-analytic set. Let $X$ be a real analytic submanifold.
A set $A\subset X$ is semi-analytic if for any $x\in \overline{A}$ (here $\overline{A}$
denotes the closure of $A$) there exists an open neighborhood $U$ of $x$ in $X$ and finitely many real-analytic functions $f_{ij}:U\to\R$, $i=1,\ldots,p$, $j=1,\ldots,q$,  such that
\begin{equation}
A\cap U=\bigcup_{i=1}^p\bigcap_{j=1}^q \{x\in U:f_{ij}(x) \ast_{ij} 0\},
\end{equation}
where the relationships $\ast_{ij}$ are either ``$>$'' or ``$=$''. A good reference for semi-analytic sets is \cite{Bierstone-Milman}.
  For example, a finite union of linear or curved polyhedra in $\R^n$, whose boundaries are level sets of real-analytic functions, is a semi-analytic set.

Now we introduce the notion of a subanalytic set, which is just obtained
in the above definition
by replacing subsets determined by inequalities 
with the ones of images of analytic maps.
That is, $A$ is said to be subanalytic if for any $x \in \overline{A}$ there exist
an open neighborhood $U$ of $x$, real analytic compact manifolds $Y_{i,j}$,
 $i=1,2,\,1\le j\le N$ and real
analytic maps $\Phi_{i,j}:Y_{i,j}\rightarrow X$ such that
\begin{equation}
A\cap U=\bigcup_{j=1}^N(\Phi_{1,j}(Y_{1,j})\setminus\Phi_{2,j}(Y_{2,j}))\,\, \bigcap\,\, U
\end{equation}
Reference is made to \cite{Bierstone-Milman} and \cite{Kashiwara-Schapira},
where we can find all the required proofs for properties stated below: A family
of subanalytic sets is stable under several set theoretical operations.
Note that, by definition, a semi-analytic subset is subanalytic.
\begin{enumerate}
\item A finite union and a finite intersection of subanalytic subsets are subanalytic.
\item The closure, interior and complement of a subanalytic subset
are again subanalytic. In particular, its boundary is subanalytic.
\item The inverse image of a subanalytic set by an analytic map is subanalytic.
Further, the direct image of a subanalytic set by a proper analytic map
is also subanalytic.
\end{enumerate}

The other important properties needed in this paper 
are the following  ``finiteness property'' and ``triangulation theorem'' of a subanalytic set.
\begin{lem}[Theorem 3.14 \cite{Bierstone-Milman}]{\label{lem:finiteness}}
	Each connected component of a subanalytic set is subanalytic.
Furthermore, connected components of a subanalytic set is locally finite, that is,
for any compact subset $K$ and a subanalytic subset $A$,
the number of connected components of $A$ intersecting $K$ is finite.
\end{lem}

In particular, for two relatively compact subanalytic subsets $A$ and $B$,
the number of connected components of $A \cap B$ is always finite.
Note that one can never expect this finiteness property if we drop subanaliticity of $A$ or $B$
as we will see in the example given at the end of this section.

\begin{lem}[Proposition 8.2.5 \cite{Kashiwara-Schapira}]{\label{lemma:triangulation}}
	Let $X= \underset{\lambda \in \Lambda}{\sqcup} X_\lambda$ be a locally finite partition
	of $X$ by subanalytic subsets. Then there exist a simplicial complex
	$\mathbf{S} = (S,\Delta)$ and a homeomorphism $i: |\mathbf{S}| \to X$ such that
\begin{enumerate}
	\item for any simplex $\sigma \in \Delta$, the image $\hat{\sigma} := i(|\sigma|)$ is subanalytic in $X$
		and real analytic smooth at every point in $\hat{\sigma}$.
\item for any simplex $\sigma \in \Delta$, there exists $\lambda \in \Lambda$
	with $i(|\sigma|) \subset X_\lambda$.
\end{enumerate}
\end{lem}

In order to define local boundary measurements we consider $\Sigma\subset\pr\dom$ an open subset of the boundary. In the context of case {\it ii)} we  will take $\Sigma$ to be smooth.

Let $\hb(\sig)$ be the closure in $H^{1/2}(\sig;\R^3)$ of the set
\begin{equation}
C_{co}(\sig)=\{\omega\in C(\pr\dom;\R^3)\;:\;\supp\omega\subset\sig\},
\end{equation}
and $\hbi(\sig)$ be its topological dual. Further, let $\tilde\dom$ be an open bounded Lipschitz (in case {\it i)}) or subanalytic (in case {\it ii)}) domain such that $\dom\subset\tilde\dom$, $\pr\dom\setminus\pr\tilde\dom=\sig$. We define 
\begin{equation}
H^1_\Sigma(\dom)=\left\{ \phi\in H^1(\dom): \exists \tilde\phi\in H^1_0(\tilde\dom), \tilde\phi|_{\dom}=\phi\right\}
\end{equation}

Our single measurement on $\Sigma$ is a pair 
\begin{equation}(\omega, (C::Du)n|_\Sigma)\in \hb(\sig)\times\hbi(\sig)
,\end{equation}
 where $n$ is the outer unit normal of $\partial\Omega$, $D=(\partial_1,\partial_2,\partial_3)^t$ formally considered as a column vector, $Du$ is the matrix $Du=(\partial_l u_k)=(\partial u_k/\partial x_l)$, and the column vector $u=(u_1,u_2,u_3)\in H^1(\Omega)$ is the solution of
\begin{equation}
L_Cu=0\quad \text{in }\dom,\quad
u|_{\pr\dom}=\omega,
\end{equation}
where 
\begin{equation}
(L_Cu)_i:=\big(\text{div}(C::Du)\big)_i:=\sum_{j=1}^3\pr_j\left( C_{ijkl}(x)\pr_l u_k\right)=0,\quad i=1,2,3,
\end{equation}
and we take the boundary condition to mean that $u-\tilde\omega\in H_0^1(\dom)$ for any $\tilde\omega\in H^1_\sig(\dom)$ such that $\tilde\omega|_{\sig}=\omega$.
The pair $(\omega, (C::Du)n|_\Sigma)$ is a Cauchy data on $\Sigma$ for the equation $L_C u=0$ in $\Omega$.

We will use infinitely many pairs of Cauchy data on $\Sigma$, namely the so called \DN map (DN map) $\dn{\sig}{C}:\hb(\sig)\to\hbi(\sig)$.
The precise definition of DN map $\dn{\sig}{C}$ is given as
\begin{equation}
\la\dn{\sig}{C}\omega,\phi\ra:=\int_\dom Du:(C::Dv)=
 \int_\dom \displaystyle\sum_{i,j,k,l=1}^3 C_{ijkl}(x)\pr_k u_l\pr_i v_j,
\end{equation}
for any $v\in H^1_\sig(\dom)$ such that $v|_{\sig}=\phi$.
Another way to put it is that, if the column vector $\vec n=(n_1,n_2,n_3)^t$ is the outer normal to the surface $\sig$, then
\begin{equation}
(\dn{\sig}{C}\omega)_i=[(C::Du)n]_i=\sum_{i,j,k,l}^3 C_{ijkl}(x)n_j\pr_k u_l|_{\sig},\quad i=1,2,3.
\end{equation}

In case {\it i)} we will consider  $C^{(I)}$, $I=1,2$, two elastic tensors which are constant on common subdomains $D_\alpha$. We assume that there is a chain $D_{\alpha_i}$, $i=1,\ldots,N$ of these subdomains (which we will abbreviate as $D_i$) and nonempty surfaces $\gam_i\subset\pr D_i$ such that $\gam_1=\sig$, and $\bar D_i\cap\bar D_{i+1}\supset\gam_{i+1}$, $i=1,\ldots,N-1$. 

If  $\gam\subset\partial D_\alpha$ is open, let $\vec n: \gam\to S^2$ be the outer normal. We will say that $\gam$ is curved if $\vec n(\gam)\subset S^2$ contains the image of a continuous curve going through two distinct points. Our main result in this case is:

\begin{thm}\label{main-thm}
Suppose that $\dn{\sig}{\Ci}=\dn{\sig}{\Cii}$, additionally suppose that each of the surfaces $\gam_i$, $i=1,\ldots, N$ is  curved, in the sense given above. Under this assumptions it follows that
$
\Ci|_{D_N}=\Cii|_{D_N}.
$
Hence if each $D_\alpha$ can be reached by such a chain,  then we have
$
\Ci=\Cii.
$
\end{thm}

In case {\it ii)} we will consider $C^{(I)}$, $I=1,2$, two elastic tensors which are constant on  subdomains $D_\alpha^{(I)}$. As in case {\it i)}, we need to impose some curvature condition on the boundaries of the subdomains $D_\alpha^{(I)}$. We will require that all the boundaries $\pr D^{(I)}_\alpha$ contain no open flat subsets. It will then follow that any open and smooth $\gam\subset \pr D^{(I)}_\alpha$ will automatically satisfy the curvature condition defined above. Our main result in this case is:
\begin{thm}\label{main-thm-2}
Under the above conditions, if $\dn{\sig}{\Ci}=\dn{\sig}{\Cii}$, then $\Ci=\Cii.$
\end{thm}

As part of the proof of the above theorems, we will also obtain the following proposition which seems interesting in its own right:

\begin{prop}\label{bd-thm}
If $\dn{\sig}{\Ci}=\dn{\sig}{\Cii}$ and $\sig$ is curved, then
$
\Ci|_{\sig}=\Cii|_{\sig}.
$
\end{prop}

\begin{remark}
More concise expressions of the anisotropic elasticity tensor exist in the case where there are symmetry axis(es) or plane(s) of symmetry (cf. \cite{ting}). In the paper \cite{nakamura-tanuma-uhlmann} on the determination of transversally isotropic heterogeneous elasticity tensor at the boundary, there are assumptions made that the axis of symmetry is known and it is uniform. What is interesting here is that we do not need to assume any symmetry.
\end{remark} 

\bigskip
Before closing this section we will make some note on the assumption and proof of the second case.
Let $\Gamma \subset \overline{\Omega}$ be 
the union of boundaries of $D^{(1)}_\alpha$ and $D^{(2)}_\beta$, i.e.,
\begin{equation}
	\Gamma := \left(\bigcup_{\alpha} \partial D^{(1)}_\alpha\right)\,\,\bigcup\,\,
	\left(\bigcup_{\beta} \partial D^{(2)}_\beta\right),
\end{equation}
and let 
\begin{equation}
	\Omega \setminus \Gamma = \underset{\lambda\in \Lambda}{\sqcup} \tilde{D}_\lambda 
\end{equation}
be the partition of $\Omega \setminus \Gamma$ by 
subanalytic subsets, where each $\tilde{D}_\lambda$ is
a connected component of $D^{(1)}_\alpha \cap D^{(2)}_\beta$ 
for some $\alpha$ and $\beta$. It follows from
Lemma {\ref{lem:finiteness}} that the number of components in this partition is finite,
that is, $\Lambda$ is a finite set.
Furthermore, as the boundary $\partial \tilde{D}_\lambda$ is contained in $\Gamma$,
the smooth surface part of $\partial \tilde{D}_\lambda$ is curved by the same assumptions
for $\partial D^{(1)}_\alpha$ and  $\partial D^{(2)}_\beta$. Note that
both the tensors $\Ci$ and $\Cii$ are constant on each $\tilde{D}_\lambda$.

Let us consider the partition 
\begin{equation}{\label{eq:partition_omega}}
\overline{\Omega} = 
\left(\underset{\lambda\in\Lambda}{\sqcup} \tilde{D}_\lambda\right) \,\sqcup\,
\Gamma
\end{equation}
of $\overline{\Omega}$. Then, by applying Lemma \ref{lemma:triangulation}
to the above partition, we obtain a
triangulation $i:|\mathbf{S}| =|(S,\Delta)| \to \overline{\Omega}$ of $\overline{\Omega}$ which is 
compatible with the partition \eqref{eq:partition_omega}.
Let $\Delta_k \subset \Delta$ ($k=0,1,2\dots$) denote the set of $k$-simplexes in $\Delta$,
and set $\Delta_{\le k} := \cup_{j \le k} \Delta_j$ for convenience.
Set also 
$$
\hat{\Omega}_k := \underset{\sigma \in \Delta_k}{\bigcup} i(|\sigma|),
\qquad
\hat{\Omega}_{\le k} := \underset{\sigma \in \Delta_{\le k}}{\bigcup} i(|\sigma|).
$$
Note that, in our case,
$\overline{\Omega} = \hat{\Omega}_3 \sqcup \hat{\Omega}_2 
\sqcup \hat{\Omega}_{\le 1}$ holds. 
We also know that $\hat{\Omega}_2$ is a real analytic smooth surface (which may consist of
a finite number of connected components) and 
$\overline{\Omega}_{\le 1}$ is the union of
finite points and a finite number of closed segments in real analytic curves.
Hence, for any points $p$ and $q$ in $\hat{\Omega} \setminus \hat{\Omega}_{\le 1}$,
we can find a $C^\infty$-smooth curve $\ell$ in $\hat{\Omega}_3 \sqcup \hat{\Omega}_2$ which
joins $p$ and $q$. Furthermore, by a local modification of $\ell$, we may assume that
$\ell$ and the smooth surface $\hat{\Omega}_2$ intersect transversally.

With this observation, the proofs of theorems \ref{main-thm} and \ref{main-thm-2} are nearly identical and we will proceed with proving them together below. 
In case {\it ii)}, for any $\tilde D_\gamma$, we will choose a chain $\tilde D_{\gamma_1},\ldots,\tilde D_{\gamma_N}$ intersecting the path $\ell$ constructed above, which we will also abbreviate as $D_1,\ldots, D_N$, such that consecutive elements are adjacent, $\sig\subset\bar D_1$, and $D_N=\tilde D_\gamma$. By our curvature condition, we can also pick common boundary patches in between successive chain elements  $\gam_i\subset\pr D_i$ such that $\gam_1=\sig$, and $\bar D_i\cap\bar D_{i+1}\supset\gam_{i+1}$, $i=1,\ldots,N-1$, and we can choose them such that they are smooth and curved.

To understand why this type of procedure cannot work in the case of Lipschitz subdomains, consider the following to subsets in the plane
\begin{equation}
A:=\left\{  (x,y):-1<x<1, x^2\sin\frac{1}{x}<y<2 \right\},
\end{equation}
\begin{equation}
B:=\left\{ (x,y):-1<x<1, -1<y<0\right\}.
\end{equation}
It is easy to see that $A\cap B$ is the union of \emph{infinitely} many disjoint open sets. Another simple example (in the plane) of what may go wrong when taking the intersection of Lipschitz sets is the following
\begin{equation}
A':=\left\{ (x,y): 1\leq x <2, 0<y<2\right\}\cup\left\{(x,y): 0< x <1, \sqrt{1-x^2}<y<2\right\},
\end{equation}
\begin{equation}
B':=\left\{(x,y): 0<x<1, 0<y<1\right\}.
\end{equation}
Then
\begin{equation}
A'\cap B'=\left\{(x,y):0< x< 1, \sqrt{1-x^2}<y<1\right\}
\end{equation}
and this set does not have the segment property at $(0,1), (1,0)\in\pr(A'\cap B')$. We therefore in case {\it i)} need to assume that  $\Ci$ and $\Cii$ have the same subdomains of homogeneity in order to prove uniqueness.

\section{Uniqueness at the boundary}\label{section-boundary}
 
 Here we will prove Proposition \ref{bd-thm}. We will make use of the Stroh formalism which we will not develop in full here as it has been extensively covered in many works. A good reference for it is \cite{tanuma}, and we will use essentially the same notations. Particularly we wish to exploit properties of the \emph{surface impedance tensor} $Z$. A reader not closely familiar with this formalism may take equation \eqref{z-formula} below as a definition of $Z$.
 
 Following \cite[Section 1.2]{tanuma}, given $\vec m, \vec n\in S^2$, and an elastic tensor $C$ we introduce  the matrices  $R(\vec m, \vec n)$, $T(\vec m, \vec n)$, defined by
 \begin{equation}
 R_{ik}=\sum_{j,l=1}^3 C_{ijkl}m_jn_l, \quad
T_{ik}=\sum_{j,l=1}^3 C_{ijkl}n_jn_l.
 \end{equation}
 These as well as the following quantities constructed from them are defined pointwise in $\dom$.
 
 We will consider $S^2$ as a real-analytic manifold with the structure induced by the atlas
 \begin{equation}
 \begin{array}{c}
 \Phi_{i,\pm}:S^2\cap\{\pm x_i>0\}\to B_{\R^2}(1),\quad i=1,2,3,\\[5pt]
 \Phi_{i,\pm}(\vec n)=(n_j, n_k), \quad j,k\neq i, j<k. 
\end{array}
  \end{equation} 
 $R$ and $T$ are then clearly real-analytic in $\vec m, \vec n\in S^2$. As per \cite[Lemma 1.1]{tanuma},  $T$ is symmetric and positive definite. Then  $T(\vec n)^{-1}$ is  real-analytic in $\vec n\in S^2$. 
 
 \begin{lem}\label{lem-analytic}
 The surface impedance tensor $Z(\vec l)$ is real-analytic in $\vec l\in S^2$.
 \end{lem}
 \begin{proof}
  Let $\vec l=(l_1,l_2,l_3)^t\in S^2$. Without loss of generality we may assume $l_3>0$, otherwise the argument given bellow may be given in a different coordinate chart. Let 
  \begin{equation}
  \vec m_0=\frac{1}{\sqrt{1-l_1^2}}(1-l_1^2,-l_1l_2,-l_1l_3)^t, \quad
  \vec n_0 = \frac{1}{\sqrt{1-l_1^2}} (0, l_3, -l_2)^t.
  \end{equation}
 and for an angle $0\leq\phi\leq2\pi$ let (see \cite[eq. (1.31)]{tanuma} for comparison)
 \begin{equation}
 \vec m_\phi=\vec m_0\cos \phi+\vec n_0\sin \phi,\quad
 \vec n_\phi=-\vec m_0\sin \phi+\vec n_0\cos \phi.
 \end{equation}
 As in \cite[Definition 1.13]{tanuma}, we consider the matrices
 \begin{equation}
 S_1(\vec l)=-\frac{1}{2\pi}\int_0^{2\pi}T(\vec n_\phi)^{-1}R(\vec m_\phi, \vec n_\phi)^t d \phi,\quad
 S_2(\vec l)=\frac{1}{2\pi}\int_0^{2\pi}T(\vec n_\phi)^{-1}d \phi.
 \end{equation}
 These two are both real-analytic in $(l_1, l_2)\in B_{\R^2}(1)$. As per \cite[Lemma 1.14]{tanuma}, $S_2$ is positive definite and hence $S_2^{-1}$ is also real analytic. 
 
 According to \cite[Theorem 1.18]{tanuma}, 
 \begin{equation}\label{z-formula}
 Z(\vec l)=S_2(\vec l)^{-1}+\sqrt{-1}S_2(\vec l)^{-1} S_1(\vec l),
 \end{equation}
  so we have our claim. 
 \end{proof}
 
 \begin{lem}
 If $P\in\sig$, we can recover the surface impedance tensor $Z(P,\vec l)$ from $\dn{\sig}{C}$, for any $\vec l\perp\vec n(P)$, $\vec l\in S^2$.\footnote{Here $C$ stands for either $\Ci$ or $\Cii$.}
 \end{lem}
 \begin{proof}
 This is a known result. See for example \cite[Theorem 2.16]{tanuma}.
 \end{proof}
 
Since $C$ is constant in $D_1$, $Z(\cdot,\vec l)$ is also constant in $D_1$. We will then denote it by $Z_{D_1}(\vec l)$. The previous lemma implies that we can determine from $\dn{\sig}{C}$ the values $Z_{D_1}(\cdot)$ takes on a circle of radius 1 situated in the plane tangent to $\sig$ at $P$. We can apply this observation to all the points of $\sig$ and, since $\sig$ is assumed to be curved, we have:

\begin{lem}\label{bd-det}
There is an open set $\mathcal{S}\subset S^2$  such that $Z_{D_1}|_{\mathcal{S}}$ can be recovered from $\dn{\sig}{C}$.
\end{lem}

From the representation of the surface impedance tensor in \cite[Thm. 1.18]{tanuma} and the definitions of the matrices involved in its representation, $Z_{D_1}(\cdot)$ is real-analytic on $S^2$. Since $\dn{\sig}{\Ci}=\dn{\sig}{\Cii}$, by the previous lemma and the unique continuation property of real-analytic functions it follows that
$Z^{(1)}_{D_1}(\vec l)=Z^{(2)}_{D_1}(\vec l)$, for all $\vec l\in S^2$.

\begin{lem}
If $C$ is a constant elastic tensor and $Z$ on $S^2$ is its associated surface impedance tensor, then the components of $C$ can be recovered from the knowledge of $Z$.
\end{lem}
\begin{proof}
The fundamental solution $\Gamma(x)=(\Gamma_{lm}(x))$ is by definition a solution of
\begin{equation}
\sum_{j,k,l=1}^3 C_{ijkl}\partial_j\partial_k\Gamma_{lm}(x)=-\delta_{im}\delta(x),\quad i,\,m\in\{1,2,3\},
\end{equation}
and has the representation formula (see \cite{nakamura-tanuma}, \cite[Thm. 2.2]{tanuma})
\begin{equation}
\Gamma(x)=(4\pi |x|)^{-1}\left(Re\; Z({x}/{|x|})\right)^{-1},\,\,x\not=0.
\end{equation}
Since we have that $C_{ijkl}\xi_j\xi_k\hat \Gamma_{lm}(\xi)=\delta_{im},\,\,i,\,m\in\{1,2,3\}$, we can then recover the elastic tensor components: 
$\sum_{j,k=1}^3 C_{ijkl}\xi_j\xi_k=(\hat\Gamma^{-1})_{il}(\xi)$, $\xi=(\xi_1,\xi_2,\xi_3)\in{\Bbb R}^3$, $i,l\in\{1,2,3\}$
from which we can recover $C=(C_{ijkl})$.
\end{proof}

\bigskip
It follows that under the assumptions of Proposition \ref{bd-thm}, $\Ci|_{D_1}=\Cii|_{D_1}$.

\section{Proof of Theorems \ref{main-thm} and \ref{main-thm-2}}\label{section-proofs}

The major task of this section is to prove the following proposition, in either case {\it i)} or case {\it ii)}.

\begin{prop}\label{keyProp}
If
$
\Lambda_{C^{(1)}}^{\Gamma_1}=\Lambda_{C^{(2)}}^{\Gamma_1},
$
then 
$\Lambda_{C^{(1)}}^{\Gamma_i}=\Lambda_{C^{(2)}}^{\Gamma_i}$, $i=2,\cdots,N$.
\end{prop}
Combining this with Lemma \ref{bd-det}, we immediately have Theorem \ref{main-thm}.

In order to prove Proposition \ref{keyProp}, we will proceed inductively. Let $\dom_i:=\dom\setminus\cap_{j=1}^{i-1}\bar D_j$. In this case $\gam_i\subset\pr\dom_i$ and we can define $\dn{\gam_i}{C}$, the \DN map for the anisotropic elasticity equation on the domain $\dom_i$, with elastic tensor $C|_{\dom_i}$, localized to $\gam_i$. To conclude the proof it is enough to show:

\begin{prop}\label{inner extension}
$\dn{\gam_1}{C}$ determines $\dn{\gam_2}{C}$. 
\end{prop}

\subsection{Discussion of notation}

Before proceeding with the proof we will clarify the assumptions we use and introduce some notations. First consider case {\it i)}. We assume $\dom_2\subset \dom_1\subset{\Bbb R}^3$ are Lipschitz domains, $D_1=\dom_1\setminus\bar\dom_2$ is connected and not empty. We define and $\Sigma_1:=\partial\dom_1\setminus\partial\dom_2$, $\Sigma_2:=\partial\dom_2\setminus\partial\dom_1$. $\gam_1\subset\sig_1$ and $\gam_2\subset\sig_2$ are non-empty, open, and curved in the sense of case {\it i)}.  We will assume we have elastic tensors $C^0$ and $C=C^0+\chi_{\dom_2}C^1$, that are piecewise constant on $\bar\dom_1$. We will use the \DN maps $\dn{\gam_1}{\Cz}$ and $\dn{\gam_1}{C}$ maps defined for the elasticity problem in the domain $\dom_1$ and $\dn{\sig_2}{C}$  defined for the elasticity problem in $\dom_2$. 
 Of course, knowledge of $\dn{\sig_2}{C}$ implies the knowledge of $\dn{\gam_2}{C}$. 

In case {\it ii)} we would like to change our notation a little. Let $P_1\in \gam_1$, $P_2\in \gam_2$. Let $g:[0,1]\to\bar D_1$ be a smooth path such that $g(0)=P_1$, $g(1)=P_2$, $g((0,1))\subset D_1$, $g'(0)\perp\gam_1$, $g'(1)\perp\gam_2$. For $\epsilon>0$ we define a tubular neighborhood $V_\epsilon$ of this path by
\begin{equation}
V_\epsilon=\{ x\in D_1: dist(x,g([0,1]))<\epsilon\}.
\end{equation}
We may clearly choose $g$ and $\epsilon$ so that 
\begin{equation}
\bar V_\epsilon\cap\gam_1=\{x\in\gam_1:dist(x,P_1)\leq\epsilon\},
\end{equation}
\begin{equation}
\bar V_\epsilon\cap\gam_2=\{x\in\gam_2:dist(x,P_2)\leq\epsilon\},
\end{equation}
and also that $\bar V_\epsilon\setminus(\gam_1\cup\gam_2)\subset D_1$ and $ V_\epsilon$ has Lipschitz boundary. 

Without loss of generality, we can redefine $\gam_1$ and $\gam_2$ to be   $\{x\in\pr D_1:dist(x,P_1)<\epsilon\}$ and $\{x\in\pr D_1:dist(x,P_2)<\epsilon\}$ respectively. We will also redefine $D_1$ to be $V_\epsilon$ and $\dom_2=\dom_1\setminus\bar V_\epsilon$. In this case $\sig_1=\gam_1$ and $\sig_2$ is Lipschitz and contains $\gam_2$.

Finally, in either case it is important to remark here that the unique continuation property of solutions ({\it UCP}) holds for
$L_{C^0}$ and $L_{C}$ using the Holmgren uniqueness theorem and regularity
up to the boundary and interfaces.

\subsection{Arguments and lemmas}

The inner extension of DN map has been already discussed in \cite{ikehata} for the conductivity equation when $\overline{\Omega_2}\subset\Omega_1$. We can easily adap the argument given there to our case by modifying the so called Runge's approximation. The arguments will be divided into several lemmas. 

For any $F\in H^{-1}(\dom_1)=(H_0^1(\dom_1))^\ast$ we define $G_0F:=U_0$, where $U$ is such that
\begin{equation}
L_\Cz U_0=-F\quad \text{in }\dom_1,\quad
U_0\in H_0^1(\dom_1).
\end{equation}
We also define $GF=G_0F+W(F)$ where 
\begin{equation}
L_C W=-L_{\chi_{\dom_2}C^1}U_0\quad \text{in }\dom_1,\quad
W\in H_0^1(\dom_1).
\end{equation}
Note that if Green's functions existed for the operators $L_\Cz$, $L_C$ on $\dom_1$, they would be the Schwartz kernels of the operators $G_0$ and $G$.

For $f\in\hbi(\sig_2)$ we define $T_f\in H^{-1}(\dom_1)$ by 
\begin{equation}
T_f(\phi)=\la f,\phi|_{\sig_2}\ra,\quad\forall\phi\in H^1_0(\dom_1).
\end{equation}
Note that is ok since $\phi|_{\sig_2}\in\hb(\sig_2)$.

We define the single layer operator $S^{\sig_2}$ as a bounded linear operator $S^{\sig_2}:\hbi(\sig_2)\to\hb(\sig_2)$ defined by
$S^{\sig_2} f=U|_{\sig_2}$, where
\begin{equation}
L_C U=-T_f\quad \text{in }\dom_1,\quad
U\in H_0^1(\dom_1).
\end{equation}
Note that that for any $f,h\in\hbi(\sig_2)$ we have
\begin{equation}
\la h, S^{\sig_2}f\ra=T_h(U)=\la h, (G T_f)|_{\sig_2}\ra=T_h(GT_f).
\end{equation}



\begin{lem}\label{runge}
Let $u\in H^1_0(\dom_1)$ be such that $L_\Cz u=0$ in an open set $V$, $\dom_2\subset V\subset\dom_1$, then there exists a sequence $u_l\in H_{\gam_1}^1(\dom)$ such that $L_\Cz u_l=0$ in $\dom_1$, $\supp u_l|_{\partial \dom_1}\subset\gam_1$, $u_l\to u$ in $H^1(\dom_2)$.
\end{lem}
\begin{proof}
Let $\dom_0$ be an open domain such that $\dom_1\subset\dom_0$, $\pr\dom_1\setminus\pr\dom_0=\gam_1$.

For $F\in H^{-1}(\dom_0)$ let $G_{\dom_0}F=u_0$, where $u_0\in H^1_0(\dom_0)$ solves $L_\Cz u_0=-F$ in $\dom_0$. We define two subspaces of $H_{\sig_2}^1(\dom_2)$:
\begin{equation}
X:=\left\{v|_{\dom_2}:v\in H^1_0(\dom_1), L_\Cz v=0\text{ in }V   \right\}
,
\end{equation}
\begin{equation}
Y:=\left\{ G_{\dom_0} F|_{\dom_2}: F\in H^{-1}(\dom_0),\supp F\subset\dom_0\setminus\overline{\dom_1}  \right\}
.
\end{equation}
It is enough to show that $Y$ is dense in $X$ with respect to $H_{\sig_2}^1(\dom_2)$. By the Hahn--Banach theorem, this is equivalent to showing that if $f\in (H_{\sig_2}^1(\dom_2))^\ast$ is such that $f(G_{\dom_0}F|_{\dom_2})=0$ for all $F\in H^{-1}(\dom_0)$, $\supp F\subset\dom_0\setminus\overline{\dom_1}$, then $f=0$ on $X$.

Define $\tilde f\in H^{-1}(\dom_0)$ by $\tilde f(\phi)=f(\phi|_{\dom_2})$ for any $\phi\in H^1_0(\dom_0)$. For any $F$ as given right above, we have
\begin{multline}
0=f(G_{\dom_0}F|_{\dom_2})=\tilde f(G_{\dom_0}F)
=\int_{\dom_0} D(G_{\dom_0}\tilde f):\big(\Cz::D(G_{\dom_0}F)\big)\\=F(G_{\dom_0}\tilde f).
\end{multline}
It then follows that $G_{\dom_0}\tilde f=0$ in $\dom_0\setminus\overline{\dom_1}$. Since $L_\Cz G_{\dom_0}\tilde f=0$ in $\dom_0\setminus\overline{\dom_2}$, by the unique continuation property for $L_\Cz$ we have that $G_{\dom_0}\tilde f=0$ in $\dom_1\setminus\bar\dom_2$.

Let $v|_{\Omega_2}\in X$ and let $\tilde v\in H^1_0(\dom_0)$ be the zero extension of $v$ to $\dom_0$. Then
\begin{multline}
f(v|_{\dom_2})=\tilde f(\tilde v)=\int_{\dom_0}D(G_{\dom_0}\tilde f):(\Cz::D\tilde v)\\
=\int_{\dom_2}D(G_{\dom_0}f):(\Cz::Dv)=0.
\end{multline}
\end{proof}

\begin{lem}\label{prop2}
Assuming $\Cz$ and $\dn{\gam_1}{C}$ are known, then $H(GF)$ can be determined for any $F,H\in H^{-1}(\dom_1)$ with $\supp F,\supp H\subset\dom_1\setminus\overline{\dom_2}$.
\end{lem}
\begin{proof}
Let $U_0,V_0\in H^1_0(\dom_1)$ be such that $L_\Cz U_0=-F$, $L_\Cz V_0=-H$. By lemma \ref{runge} there exist $U_l,V_l\in H^1_0(\dom_1)$, $\supp U_l|_{\partial\dom_1},\supp V_l|_{\partial\dom_1}\subset\Gamma_1$ such that $U_l\to U_0$, $V_l\to V_0$ in $H_{\sig_2}^1(\dom_2)$.

Let $Z_l\in H^1(\dom_1)$ be the solution of 
\begin{equation}
L_C Z_l=0\quad \text{in }\dom_1,\quad
Z_l|_{\partial\dom_1}=U_l|_{\partial\dom_1}.
\end{equation}
and let $W_l:=Z_l-U_l\in H^1_0(\dom_1)$. Then $L_C W_l=-L_{\chi\dom_2C^1}U_l$ in $\dom_1$ and $W_l\to W=GF-G_0F$ in $H^1_0(\dom_1)$.

We have that 
\begin{multline}
\la (\dn{\gam_1}{C}-\dn{\gam_1}{\Cz})U_l|_{\partial\dom_1},V_l|_{\partial\dom_1}\ra
=\int_{\dom_1} DZ_l:(C-\Cz):DV_l\\
=\int_{\dom_2} DU_l:C^1:DV_l+\int_{\dom_2}DW_l:C^1:DV_l\\
\to \int_{\dom_2} DU_0:C^1:DV_0+\int_{\dom_2}DW:C^1:DV_0.
\end{multline}
On the other hand
\begin{multline}
H(W)=\int_{\dom_1} DV_0:\Cz:DW
=\int_{\dom_1} DV_0:(C-\chi_{\dom_2}C^1):D W\\
=\int_{\dom_1} DV_0:C:DW-\int_{\dom_2}DV_0:C^1:DW\\
=-\int_{\dom_2} DU_0:C^1:DV_0-\int_{\dom_2} DV_0:C^1:DW.
\end{multline}
Therefore
\begin{equation}
H(GF-G_0F)=\lim_{l\to\infty}\la (\dn{\gam_1}{C}-\dn{\gam_1}{\Cz})U_l|_{\partial\dom_1},V_l|_{\partial\dom_1}\ra.
\end{equation}
\end{proof}

\begin{lem}\label{s-det}
Assuming $\Cz$ and $\dn{\gam_1}{C}$ are known, we can determine $S^{\sig_2}$.
\end{lem}
\begin{proof}
Since $L_{co}^2(\Sigma_2):=\{f\in L^2(\partial\Omega_2):\,\text{supp}f\subset\Sigma_2\}$ is dense in $H_{co}^{-1/2}(\Sigma_2)$ and $S^{\Sigma_2}$
is bounded linear, we only need to show that $S^{\Sigma_2}f$ can be computed for 
any $f\in L_{co}^2(\Sigma_2)$.
Around any point $P\in\sig_2$ we can find coordinates such that locally $\sig_2$ is $\{x_3=0\}$ and $\dom_1\setminus\overline{\dom_2}$ is $\{x_3>0\}$. Let $f\in L^2(\sig_2)$ be supported in this local coordinate patch. Define $F_\epsilon\in H^{1}(\dom_1)$ by 
\begin{equation}
F_\epsilon(\phi)=\la f,\phi_{x_3=\epsilon}\ra=\int f(x')\phi(x',\epsilon) d x',
\end{equation}
for any $\phi\in H^1_0(\dom_1)$. Then
\begin{equation}
|F_\epsilon(\phi)-T_f(\phi)|=
=\left|\int_{0\leq x_3\leq\epsilon} f(x')\partial_3\phi(x',x_3)\right|
\leq ||\phi||_{H^1_0(\dom_1)}||f||_{L^2(\sig_2)}\epsilon^{1/2}.
\end{equation}
Using a partition of unity argument we have that for any $f\in L^2(\partial\dom_2)$ with $\supp f\subset\sig_2$ we can construct $F_n\in H^{-1}(\dom_1)$ such that $F_n\to T_f$ in $H^{-1}(\dom_1)$.

Let $f,\,h \in L_{co}^2(\Sigma_2)$. Since we have that 
$\la h, S^{\sig_2}f\ra=T_f(GT_h)$ and we can approximate $T_f$ and $T_h$ by $F_n$, $H_n$ with $\supp F_n, \supp H_n\subset\dom_1\setminus\overline{\dom_2}$, by Lemma \ref{prop2} we obtain the desired conclusion.\qedhere
\end{proof}

Now let $\dn{\sig_2,+}{C}$ be the local \DN map associated to the elasticity problem in $\dom_1\setminus\overline{\dom_2}$ with boundary data supported in $\sig_2$. That is, if $f\in\hb(\sig_2)$ and $u^+\in H^1(\dom_1\setminus\overline{\dom_2})$ is the solution of 
\begin{equation}
L_C u^+=0\quad \text{in }\dom_1\setminus\bar\dom_2,\quad
u_+|_{\partial(\dom_1\setminus\bar\dom_2)}=f,
\end{equation}
then for any $v\in H^1(\dom_1\setminus\overline{\dom_2})$ such that $v|_{\partial(\dom_1\setminus\overline{\dom_2})}=h\in\hb(\sig_2)$, 
\begin{equation}
\la\dn{\sig_2,+}{C}f, h\ra=-\int_{\dom_1\setminus\overline{\dom_2}}Du^+:C:Dv.
\end{equation}

\begin{lem}\label{formulae for DN maps}${}$
\newline
{\rm (i)} $\dn{\sig_2}{C}-\dn{\sig_2,+}{C}$ is injective.
\newline
{\rm (ii)} For any $f\in\hbi(\sig_2)$, $(\dn{\sig_2}{C}-\dn{\sig_2,+}{C})S^{\sig_2}f=f$.
\end{lem}
\begin{proof}
Although the proof is almost the same as in \cite{ikehata}, we will provide it to clarify that the DN maps can be localized here. We first prove (i). To begin with let $f\in\hb(\sig_2)$ be such that $(\dn{\sig_2}{C}-\dn{\sig_2,+}{C})f=0$. If $u\in H^{1}(\dom_2)$ solves
\begin{equation}
L_C u=0\quad \text{in }\dom_2,\quad
u|_{\partial\dom_2}=f,
\end{equation}
and $u^+$ is as above,  let
$
\tilde u:= \chi_{\dom_2}u+\chi_{\dom_1\backslash\dom_2}u^+.
$
Then $\tilde u\in H^1_0(\dom_1)$ and $L_C u=0$ in the whole $\dom_1$. It follows that $\tilde u\equiv 0$.

Next we prove (ii). For that let $\phi\in H^1_0(\dom_1)$. Then
\begin{multline}
\la f,\phi|_{\sig_2}\ra
=\int_{\dom_1-\bar\dom_2}D(GT_f):C:D\phi+\int_{\dom_2} D(GT_f):C:D\phi\\
=-\la \dn{\sig_2,+}{C}(S^{\sig_2}f),\phi|_{\sig_2}\ra+\la \dn{\sig_2}{C}(S^{\sig_2}f),\phi|_{\sig_2}\ra\\
=\la(\dn{\sig_2}{C}-\dn{\sig_2,+}{C})S^{\sig_2}f,\phi|_{\sig_2}\ra
\end{multline}
\end{proof}

From Lemma \ref{formulae for DN maps},  it follows that $\dn{\sig_2}{C}-\dn{\sig_2,+}{C}=(S^{\sig_2})^{-1}$ and hence combining this with Lemma \ref{s-det} and the observation that $\dn{\sig_2}{C}$ determines $\dn{\gam_2}{C}$, we have Proposition \ref{inner extension}.

\bigskip
\noindent
{\bf Acknowledgement} The work of this paper was initiated when the third author stayed in NCTS (National Center for Theoretical Science) of National Taiwan University, Taipei for 3 months with support during the year of 2016. Also he was partially supported by grant-in-aid for Scientific Research (15K21766 and 15H05740) of the Japan Society for the Promotion of Science for inviting the first author to Japan in order to work on the subject of this paper.  Subsequently the first author has traveled to Hokkaido University for 11 days, when further work was done towards this paper. The transportation costs were covered by NCTS, local support provided by Hokkaido University. We acknowledge all these support.

\end{document}